\documentclass[letterpaper, 11pt]{amsart}
 
\usepackage{fullpage}
\usepackage{graphicx}
\usepackage{amsmath, amssymb, amsfonts, amsthm}
\usepackage{url}

\usepackage[ruled,vlined]{algorithm2e}

\usepackage{color}
\definecolor{darkblue}{rgb}{0,0,0.4}
\usepackage[plainpages=false,colorlinks=true,citecolor=blue,filecolor=black,linkcolor=red,urlcolor=darkblue]{hyperref}

\newtheorem{thm}{Theorem}[section]

\newtheorem{lem}[thm]{Lemma}
\theoremstyle{remark}
\newtheorem{rem}[thm]{Remark}

\numberwithin{equation}{section}

\usepackage{mathtools}
\usepackage{newlfont}
\usepackage{graphicx}
\usepackage{float}
\usepackage{geometry}

\usepackage{caption}
\usepackage{subcaption}
\usepackage{comment}

\newcommand{\eps}{\varepsilon}

\newcommand{\ed}{\end {document}}

\newcommand{\diag}{\mathrm{diag}}

\graphicspath{{./figs/}} 

\usepackage[backend=bibtex,giveninits=true,maxcitenames=2,maxbibnames=10,style=numeric,natbib=true,url=false,sorting=nyt,doi=true]{biblatex}
\renewbibmacro{in:}{\ifentrytype{article}{}{\printtext{\bibstring{in}\intitlepunct}}}

\begin{document}

\title{Unconditional energy dissipation of Strang splitting for the matrix-valued Allen--Cahn equation}

\author{Chaoyu Quan}
\address{School of Science and Engineering, The Chinese University of Hong Kong, Shenzhen, Guangdong 518172, China}
\email{quanchaoyu@cuhk.edu.cn}

\author{Tao Tang}
\address{School of Mathematics and Statistics, Guangzhou Nanfang College, Guangzhou, Guangdong Province, China}
\email{ttang@nfu.edu.cn}

\author{Dong Wang}
\address{School of Science and Engineering, The Chinese University of Hong Kong, Shenzhen, Guangdong 518172, China \& Shenzhen International Center for Industrial and Applied Mathematics, Shenzhen Research Institute of Big Data, Guangdong 518172, China}
\email{wangdong@cuhk.edu.cn}

\subjclass[2010]{65M12, 65M15} 

\keywords{Matrix-valued Allen--Cahn equation, energy dissipation law, Strang splitting, stability and convergence}

\date{\today}

\begin{abstract} 
The energy dissipation property of the Strang splitting method was first demonstrated for the matrix-valued Allen–Cahn (MAC) equation under restrictive time-step constraints [J. Comput. Phys. 454, 110985, 2022]. In this work, we eliminate this limitation through a refined stability analysis framework, rigorously proving that the Strang splitting method preserves the energy dissipation law unconditionally for arbitrary time steps. The refined proof hinges on a precise estimation of the double-well potential term in the modified energy functional. Leveraging this unconditional energy dissipation property, we rigorously establish that the Strang splitting method achieves global-in-time $H^1$-stability, preserves determinant boundedness, and maintains second-order temporal convergence for the matrix-valued Allen–Cahn equation. To validate these theoretical findings, we conduct numerical experiments confirming the method’s energy stability and determinant bound preservation for the MAC equation.
\end{abstract}

\maketitle

\section{Introduction}
The theory of singular perturbation for scalar-valued phase transition problems has attracted considerable attention in both analytical and computational studies, due to its wide-ranging applications across various scientific fields and its inherent mathematical challenges. A classic example is phase transition in the set $\{+1, -1\}$, which typically leads to the Allen--Cahn reaction-diffusion equations in the $L^2$ gradient flow, and in the set $S^1$, resulting in the complex Ginzburg-Landau equation. More general target sets composed of two disconnected manifolds have also been explored; see, for instance, \cite{rubinstein1989reaction,lin2012phase,liu2022phase} and the references therein.

The study is broadly categorized into two main parts: the gradient flow system and the minimization of the energy functional. Gradient flow systems are commonly employed to model the dynamics of interfaces. For instance, the scalar Allen--Cahn equation captures the behavior of two-phase fluid problems, approximately modeling the mean curvature flow of the interface. Stationary solutions to such problems typically emerge from the task of finding harmonic maps from a domain to a specified target set. These solutions find significant applications in various fields, including inverse problems in image analysis for diffusion tensor MRI or fiber tractography, where estimating matrix- or orientation-valued functions is of interest \cite{osting2020diffusion2,bacak2016second}. Additionally, they play a pivotal role in directional field synthesis problems in geometry processing and computer graphics, which are often formulated as the problem of identifying matrix-valued fields within a specific class \cite{golovaty2021variational,viertel2019approach}.

In this paper, we focus on the numerical methods for the case where phase transition takes in the set of orthogonal matrix group $O(m) = O^+(m) \cup O^-(m)$ where $O^{\pm}(m)$ denotes the set of orthogonal matrices with determinate $\pm 1$. This leads to  the matrix-valued Allen--Cahn (MAC) equation introduced in \cite{osting2020diffusion}. The equation governing the evolution of a matrix field \( U(t,x): [0,\infty) \times \Omega \to \mathbb{R}^{m \times m} \) is given by
\begin{equation} \label{eq:matrixAC}
\partial_t U = \varepsilon^2 \Delta U  + U - UU^{\top}U,
\end{equation}
where \(\Omega = [-\pi, \pi]^d\) is a \(2\pi\)-periodic torus and \(\varepsilon > 0\) characterizes the interface thickness. It describes the gradient flow of a Ginzburg--Landau energy that penalizes deviations from orthogonality. Its dynamics reveal rich interfacial phenomena: at a fast time scale, interfaces evolve by mean curvature motion, while at a slower time scale, matrix-valued surface diffusion dominates \cite{wang2019interface}. These properties make the MAC equation a potential tool in materials science and geometric analysis.

Numerical methods for the MAC equation shall preserve two key properties: the \textit{maximum bound principle} (MBP), ensuring \(\|U(t,x)\|_F \leq \sqrt{m}\), and the \textit{energy dissipation law}. While the scalar Allen-Cahn equation has inspired schemes like exponential time differencing (ETD) \cite{du2021maximum} and integrating factor Runge--Kutta (IFRK) methods \cite{sun2023maximum}, their matrix-valued extensions often impose restrictive symmetry assumptions on initial data to ensure MBP. For example, ETD schemes in \cite{du2021maximum} and IFRK methods in \cite{sun2023maximum} require \(U(0,x)\) to be symmetric—a condition not generally satisfied. Recently, such symmetry assumption is removed for the ETD schemes \cite{liu2024maximum}. Operator splitting methods, such as the Strang splitting, avoid this symmetry assumption but initially demanded time-step constraints for energy dissipation \cite{li2022stability2}. Concurrently, theoretical advances, including the sharp-interface limit analysis in \cite{fei2023matrix}, have deepened understanding of the MAC equation’s geometric properties.


In this work, we improve the energy dissipation analysis of the Strang splitting method for \eqref{eq:matrixAC} in \cite{li2022stability2}, by proving that the energy dissipation law holds \textit{unconditionally}—without time-step restrictions. 
Additionally, we establish the scheme’s uniform $H^1$-stability, uniform bound of determinant, and second-order convergence in time. The highlights of this work include:
\begin{itemize}
    \item[(1)] {Unconditional energy dissipation:} Based on new matrix analysis techniques, we eliminate the time-step dependencies in the energy stability proof (Theorem \ref{thm:energy}), which improves the energy stability result under small time-step constraint in \cite{li2022stability2};
    \item[(2)] {Uniform stability and convergence analysis:} A priori \(H^1\)-stability estimates (Theorem \ref{thm:stb}) and determinant bounds (Theorem \ref{thm:det}) enable a full second-order error analysis (Theorem \ref{thm:conv}), extending results from the scalar case;
    \item[(3)] {Connection to diffusion-generated methods:} As \(\varepsilon \to 0\), the scheme reduces to a diffusion-generated method for matrix fields, bridging phase-field and threshold dynamics approaches.
\end{itemize}
The unconditional energy stability allows larger time steps in practical computations, while the convergence guarantees reliability. 

The rest of this paper is organized as follows:
Section \ref{sec:strang_operator_splitting} introduces the Strang operator splitting scheme for the matrix-valued Allen--Cahn equation, including the unconditional energy dissipation property.
Section \ref{sec:stability_convergence} provides a detailed analysis of the global-in-time $H^1$-stability and uniform bound of the determinant for the numerical solution.
In addition, we present the second-order convergence in time for the Strang splitting method based on the stability results.
Section \ref{sec:connections} discusses the connections to diffusion-generated methods, including the optimization perspective and the case when $\varepsilon\rightarrow 0$. In Section \ref{sec:num}, several numerical tests are conducted to illustrate the energy stability and the determinant bound preservation of Strang splitting method for MAC equation, as well as its comparison with thresholding method.
Finally, Section \ref{sec:conclusion} concludes the paper and outlines future work.

\section{The Strang operator splitting scheme for \eqref{eq:matrixAC}}\label{sec:strang_operator_splitting}




The Strang operator splitting scheme was introduced to solve the scalar or matrix-valued Allen--Cahn equation in \cite{li2022stability,li2022stability2}. It employs a symmetric operator splitting approach, where the equation is divided into parts that are alternately solved over staggered time intervals. This symmetric structure cancels out leading-order errors, achieving second-order accuracy in time.
We consider firstly the nonlinear part of MAC equation, i.e. the following 
ODE for $U=U(t):\, [0,\infty) \to \mathbb R^{m\times m}$:
\begin{align} \label{eq:nlODE}
\begin{cases}
\partial_t U = U-UU^\top U, \\
U(0) = U^0 \in \mathbb R^{m\times m}.
\end{cases}
\end{align}
The unique smooth
 solution $U(t)$ to \eqref{eq:nlODE} is given by (see \cite{li2022stability2} for details)
\begin{align} \label{3t1}
\mathcal S_{\mathcal N}(t) U^0 := U(t) = \left((e^{2t}-1) U^0(U^0)^\top  +  I\right)^{-\frac12} e^t U^0,\qquad t>0.
\end{align}
Given $U:\, \Omega \to \mathbb R^{m\times m}$ and $t>0$, we define 
the linear propagator 
\begin{align}
\Bigl( \mathcal S_{\mathcal L} (t) U \Bigr)_{ij} (x) =  \Bigl( e^{t\varepsilon^2\Delta} U_{ij} \Bigr)(x),
\qquad i, j=1,\cdots m.
\end{align}
In other words, the operator $\mathcal S_{\mathcal L}(t) = e^{t\varepsilon^2\Delta}$ is applied to the matrix $U$ entry-wise. 
The Strang splitting method for \eqref{eq:matrixAC} is 
\begin{align}\label{eq:Strang}
U^{n+1}  = \mathcal S_{\mathcal L}\left( \tau/2\right) \mathcal S_{\mathcal N}\left( \tau \right)\mathcal S_{\mathcal L}\left( \tau/2\right) U^n.
\end{align}

In \cite{li2022stability2}, the key ingredient of establishing the energy dissipation law of Strang splitting method for MAC is a trace inequality (see Lemma 3.1--3.3 in \cite{li2022stability2}) with a restriction on time step, for example, $\tau <0.0631$ for $m=3$. 
Next, we will provide a novel framework of analysis so that such restriction on time step can be removed.  

\subsection{Unconditional energy dissipation}
We will use the following identity
\begin{align}
\langle A, \, B \rangle_F =  \mathrm{Tr}(AB^\top  ) = \sum_{i,j=1}^m A_{ij} B_{ij},
\qquad\forall\, A, B \in \mathbb R^{m\times m},
\end{align}
where $\langle \cdot,\, \cdot\rangle_F$ denotes the usual Frobenius inner product and $\mathrm{Tr}$ is the trace operator.

\begin{lem}\label{lem:ineqTr}
    For any $U\in\mathbb R^{m\times m}$ and any $\alpha>0$, 
    \begin{equation}
        \mathrm{Tr}\left[\left( I+ \alpha \mathrm{diag}^2(U)\right)^{\frac12}\right] \leq \mathrm{Tr}\left[\left( I+ \alpha U U^{\top}\right)^{\frac12}\right]. 
    \end{equation}
    In other words, the above inequality is
    \begin{equation}
        \| [I~\alpha^{\frac12}\diag(U)]\|_* \leq \| [I ~\alpha^{\frac12} U]\|_*,
    \end{equation}
    where $\|\cdot\|_*$ denotes the nuclear norm. 
\end{lem}
\begin{proof}
Let $U = \left(u_{ij}\right)_{1\leq i,j\leq m}$.
Assume that $U$ has a singular value decomposition
\begin{equation}
    U = P \mathrm{diag}(\delta_1,\delta_2,\ldots,\delta_m) Q^\top
\end{equation}
with $\delta_i\geq 0$ are the singular values of $U$, $P = \left(p_{ij}\right)\in \mathbb R^{m\times m}$ and $Q = \left(q_{ij}\right)\in \mathbb R^{m\times m}$ are orthogonal matrix.
We have
\begin{equation}
    u_{ij} = \sum_{k =1}^m \delta_{k} p_{ik}q_{jk},\quad \forall 1\leq i,j\leq m.
\end{equation}

We have
\begin{equation}\label{eq:TrUU}
     \mathrm{Tr}\left[\left( I+ \alpha U U^\top\right)^{\frac12}\right] =  \sum_{k=1}^m \left(1+\alpha \delta_{k}^2\right)^{\frac12} =  \sum_{k =1}^m h\left(  \delta_{k} \right)
\end{equation}
and 
\begin{equation}\label{eq:TrDelta}
\begin{aligned}
     \mathrm{Tr}\left[\left( I+ \alpha \mathrm{diag}^2(U)\right)^{\frac12}\right] = 
     & \sum_{i=1}^m \left(1+\alpha u_{ii}^2\right)^{\frac12} = \sum_{i=1}^m \left(1+\alpha \left(\sum_{k =1}^m \delta_{k} p_{ik}q_{ik}\right)^2\right)^{\frac12} \\
     =& \sum_{i=1}^m h \left(\sum_{k =1}^m \delta_{k} p_{ik}q_{ik}\right),
     \end{aligned}
\end{equation}
where 
\begin{equation}
    h(\lambda) = \left(1+\alpha \lambda^2\right)^{\frac12}.
\end{equation}
We compute straightly: $\forall \lambda>0,$ 
\begin{equation}
\begin{aligned}
    h'(\lambda) = \alpha \lambda (1+\alpha \lambda^2)^{-\frac12}>0\quad\mbox{and}\quad
    h''(\lambda) = \alpha (1+\alpha \lambda^2)^{-\frac32} > 0,
    \end{aligned}
\end{equation}
meaning that $h(\lambda)$ is a convex function. 
From \eqref{eq:TrUU}, \eqref{eq:TrDelta}, the facts
\begin{equation}
   \sum_{k =1}^m \frac{p_{ik}^2+q_{ik}^2}{2}=1,\quad \sum_{i =1}^m \frac{p_{ik}^2+q_{ik}^2}{2}=1,
\end{equation}
and the Jensen's inequality for convex function, we have
\begin{equation}
    \begin{aligned}
        &\mathrm{Tr}\left[\left( I+ \alpha \mathrm{diag}^2(U)\right)^{\frac12}\right] \\
         = &  \sum_{i=1}^m h\left(\sum_{k =1}^m \delta_{k} p_{ik}q_{ik}\right)
        =  \sum_{i=1}^m h\left(\left|\sum_{k =1}^m \delta_{k} p_{ik}q_{ik}\right|\right)\\
         \leq & \sum_{i=1}^m  h\left( \sum_{k =1}^m \delta_{k} \left|p_{ik}q_{ik}\right|\right)
         \leq \sum_{i=1}^m  h\left( \sum_{k =1}^m \delta_{k} \frac{p_{ik}^2+q_{ik}^2}{2}\right)\\
         \leq & \sum_{i=1}^m \sum_{k =1}^m \frac{p_{ik}^2+q_{ik}^2}{2}  h\left(  \delta_{k} \right)
        =  \sum_{k =1}^m \left(\sum_{i=1}^m \frac{p_{ik}^2+q_{ik}^2}{2}  \right) h\left(  \delta_{k} \right) = \sum_{k =1}^m h\left(  \delta_{k} \right)\\
        = & \mathrm{Tr}\left[\left( I+ \alpha U U^\top\right)^{\frac12}\right].
    \end{aligned}
\end{equation}
Finally, note that for any $X\in \mathbb R^{m\times m}$,
$$
\mathrm{Tr}\left[\left( I+ XX^\top\right)^{\frac12}\right] = \sum_{i =1}^m\sigma_i([I~ X]) = \|X\|_*.
$$
We then deduce the equivalent nuclear-norm inequality in the lemma.

\end{proof}

\begin{lem}\label{lem2}
For any $U\in\mathbb R^{m\times m}$, it holds that
    \begin{equation}
         -\frac14 m \leq \left\langle G(U),I\right\rangle_F
         \leq  \left\langle G(\mathrm{diag}(U)),I\right\rangle_F + \frac1{2\tau}\|U-\mathrm{diag}(U)\|_F^2,  
\end{equation}
where 
\begin{equation}\label{eq:funG}
    G(U)  \coloneqq \frac1{2\tau} U U^\top  - \frac{e^\tau}{\tau(e^{2\tau}-1)} \left(\left(  I+(e^{2\tau}-1)U U^\top \right)^{\frac12}-  I\right).
\end{equation}
\end{lem}
\begin{proof}
Assume that $U$ has a singular value decomposition: $U = U_0\Sigma V_0^\top$ where $U_0$ and $V_0$ are orthogonal, and $\Sigma =\diag\{\sigma_1,\sigma_2,\ldots,\sigma_m\}$ is a diagonal matrix with singular values. 
 \begin{equation}
        \begin{aligned}
        \left\langle G(U),I\right\rangle_F =&  \frac{1}{2\tau} \mathrm{Tr}\left(UU^\top\right)-  \frac{e^\tau}{\tau(e^{2\tau}-1)}  \mathrm{Tr}\left[\left( I+ (e^{2\tau}-1) U U^\top\right)^{\frac12}-I\right]  \\
        =&  \frac{1}{2\tau} \mathrm{Tr}\left(\Sigma^2\right)-  \frac{e^\tau}{\tau(e^{2\tau}-1)}  \mathrm{Tr}\left[\left( I+ (e^{2\tau}-1) \Sigma^2\right)^{\frac12}-I\right]\\
         =&  \sum_{i=1}^m \left[\frac{1}{2\tau} \sigma_i^2-\frac{e^\tau}{\tau(e^{2\tau}-1)} \left((1+(e^{2\tau}-1) \sigma_i^2)^{\frac12}-1\right) \right]\geq -\frac14 m,
         \end{aligned}
    \end{equation}
    where we use the fact that
    \begin{equation}
        \frac{1}{2\tau} s-\frac{e^\tau}{\tau(e^{2\tau}-1)} \left((1+(e^{2\tau}-1) s)^{\frac12}-1\right) \geq -\frac14,\quad \forall \tau>0,~s\geq 0.
    \end{equation}

Let $U = \left(u_{ij}\right)_{1\leq i,j\leq m}$.
According to the definition of $G(U)$ and Lemma \ref{lem:ineqTr}, we compute
    \begin{equation}
        \begin{aligned}
        & \left\langle G(U),I\right\rangle_F
         -  \left\langle G(\mathrm{diag}(U)),I\right\rangle_F 
         = \mathrm{Tr}(G(U)) -\mathrm{Tr}(G(\mathrm{diag}(U))) \\
        = & \frac{1}{2\tau} \mathrm{Tr}\left(UU^\top-\diag^2(U)\right) -  \frac{e^\tau}{\tau(e^{2\tau}-1)} \\
        &\quad \left\{ \mathrm{Tr}\left[\left( I+ (e^{2\tau}-1) U U^\top\right)^{\frac12}\right] 
        -\mathrm{Tr}\left[\left( I+ (e^{2\tau}-1) \mathrm{diag}^2(U)\right)^{\frac12}\right] \right\} \\
        \leq & \frac{1}{2\tau} \mathrm{Tr}\left(UU^\top-\diag^2(U)\right)
        =\frac1{2\tau}\|U-\mathrm{diag}(U)\|_F^2.
         \end{aligned}
    \end{equation}
\end{proof}

\begin{thm}[Unconditional energy dissipation]\label{thm:energy}
For any $\tau>0$, the solution of the Strang splitting method \eqref{eq:Strang}
satisfies the following energy dissipation property 
\begin{align}
\widetilde E({U}^{n+1})\le \widetilde E({U}^{n}),\quad \forall\, n\ge 0,
\end{align}
where 
\begin{align}
\widetilde E({U}) =\int_\Omega \frac{1}{2\tau}\left\langle (1-e^{\tau \varepsilon^2\Delta})  {U},  {U} \right\rangle_F + \left\langle G(e^{\tau/2 \varepsilon^2\Delta} U),I\right\rangle_F \, dx \label{eq:modifiedenergy}
\end{align}
and $G$ is defined in \eqref{eq:funG}.

\end{thm} 
\begin{proof}
Let $\tilde U^n = \mathcal S_{\mathcal L}\left( \tau/2\right) U^n$.
According to \eqref{eq:Strang}, we have
\begin{align}
e^{-\tau \varepsilon^2\Delta}\tilde{U}^{n+1}  = \left((e^{2\tau}-1) \tilde U^n {(\tilde U^n)}^\top  +  I\right)^{-\frac12} e^\tau \tilde U^n.
\end{align}
We rewrite the above as
\begin{align} \label{t7}
\frac 1 {\tau} (e^{-\tau \varepsilon^2\Delta}-1)\tilde{U}^{n+1} + \frac 1 {\tau}(\tilde{U}^{n+1}-\tilde{U}^{n}) = \frac 1 {\tau} \left(  {\left((e^{2\tau}-1) \tilde{U}^n {(\tilde U^n)}^\top  + I\right)^{-\frac12}} {e^\tau\tilde{U}^n}-\tilde{U}^{n} \right).
\end{align}
Taking the Frobenius inner product with  $ 
\tilde U^{n+1} - \tilde U^{n} $ on both sides of \eqref{t7}, we obtain
\begin{equation}\label{eq:idFrob}
\begin{aligned}
&\frac 1 {\tau} \langle  (e^{-\tau \varepsilon^2\Delta}-1)\tilde{U}^{n+1},\;
\tilde{U}^{n+1}-\tilde{U}^{n} \rangle_F
+ \frac 1 {\tau} \| \tilde{U}^{n+1}-\tilde{U}^{n} \|_F^2  \\
=&\frac 1 {\tau} \left\langle {\left((e^{2\tau}-1) \tilde{U}^n {(\tilde U^n)}^\top  +  I\right)^{-\frac12}} {e^\tau\tilde{U}^n}-\tilde{U}^{n} , \;\;
\tilde{U}^{n+1}-\tilde{U}^{n}  \right\rangle_F.
\end{aligned}
\end{equation}
It is not difficult to check that
\begin{equation}\label{ineq:term1}
\begin{aligned}
 & \int_{\Omega} \langle  (e^{-\tau \varepsilon^2\Delta}-1)\tilde{U}^{n+1},\;
\tilde{U}^{n+1}-\tilde{U}^{n} \rangle_F \,dx \\
=&\; \frac 12 \int_{\Omega} 
\langle  (e^{-\tau \varepsilon^2\Delta}-1)\tilde{U}^{n+1},\;
\tilde{U}^{n+1} \rangle_F\, dx -\frac 12 \int_{\Omega} 
\langle  (e^{-\tau \varepsilon^2\Delta}-1)\tilde{U}^{n},\;
\tilde{U}^{n} \rangle_F \,dx \\
& \quad +\frac 12 \int_{\Omega} 
\langle  (e^{-\tau \varepsilon^2\Delta}-1)(\tilde{U}^{n+1}-\tilde{U}^n),\;
\tilde{U}^{n+1} -\tilde{U}^n\rangle_F\, dx .
\end{aligned}
\end{equation}
Assume that $\tilde U^n$ has singular value decomposition $\tilde U^n = U_0\Sigma V_0^\top$ where $U_0$ and $V_0$ are real orthogonal matrices. 
Let 
\begin{equation}
    U_0^\top\tilde{U}^{n+1}V_0 = \left(\tilde u_{ij}\right).
\end{equation}
Then
\begin{equation}\label{eq:rhs}
    \begin{aligned}
        &\left\langle {\left((e^{2\tau}-1) \tilde{U}^n {(\tilde U^n)}^\top  +  I\right)^{-\frac12}} {e^\tau\tilde{U}^n}-\tilde{U}^{n} , \;\tilde{U}^{n+1}-\tilde{U}^{n}  \right\rangle_F \\
        =&  \left\langle U_0 \left[e^\tau\left((e^{2\tau}-1) \Sigma^2 +  I\right)^{-\frac12} -I\right] \Sigma V_0^\top, \;\tilde{U}^{n+1}-U_0\Sigma V_0^\top  \right\rangle_F \\
        =&  \left\langle \left[e^\tau\left((e^{2\tau}-1) \Sigma^2 +  I\right)^{-\frac12} -I\right] \Sigma, \; U_0^\top\tilde{U}^{n+1}V_0-\Sigma \right\rangle_F \\
        = & \sum_{i=1}^m \left[e^\tau\left((e^{2\tau}-1) \sigma_i^2 +  1\right)^{-\frac12} -1\right] \sigma_i (\tilde u_{ii}-\sigma_i) \\
         = & \sum_{i=1}^m -g'(\sigma_i) (\tilde u_{ii}-\sigma_i) =\sum_{i=1}^m -g(\tilde u_{ii})+g(\sigma_i)+\frac12 g''(\xi_i)(\tilde u_{ii}-\sigma_i)^2\\
        = &\sum_{i=1}^m -g(\tilde u_{ii})+g(\sigma_i)+\frac12 (\tilde u_{ii}-\sigma_i)^2,
    \end{aligned}
\end{equation}
where $\xi_i$ is some constant between $\tilde u_{ii}$ and $\sigma_i$,
\begin{equation}
    g(\lambda) = \frac12 \lambda^2 - \frac{e^\tau}{e^{2\tau}-1} \left[\left(1+(e^{2\tau}-1)\lambda^2\right)^{\frac12}-1\right]
\end{equation}
and 
\begin{equation}
    g''(\lambda) = -e^\tau\left(1+(e^{2\tau}-1)\lambda^2\right)^{-\frac32}+1\leq 1,\quad \forall \lambda\in\mathbb R.
\end{equation}
According to the definition of $G(U)$, we have
\begin{equation}\label{eq:Gun}
    \begin{aligned}
         \left\langle G({\tilde U}^n),I\right\rangle_F = \mathrm{Tr}(G({\tilde U}^n)) = \mathrm{Tr}(G(\Sigma)) = \left\langle G(\Sigma),I\right\rangle_F = \frac1\tau \sum_{i=1}^m g(\sigma_{i})
    \end{aligned}
\end{equation}
and
\begin{equation}\label{ineq:Gun1}
    \begin{aligned}
         &\left\langle G({\tilde U}^{n+1}),I\right\rangle_F 
         =  \left\langle G(U_0^\top\tilde{U}^{n+1}V_0),I\right\rangle_F\\
         \leq & \left\langle G(\mathrm{diag}(U_0^\top\tilde{U}^{n+1}V_0)),I\right\rangle_F + \frac1{2\tau}\|U_0^\top\tilde{U}^{n+1}V_0-\mathrm{diag}(U_0^\top\tilde{U}^{n+1}V_0)\|_F^2  \\
         = & \left\langle G(\mathrm{diag}(U_0^\top\tilde{U}^{n+1}V_0)),I\right\rangle_F + \frac1{2\tau} \sum_{i\neq j} \tilde u_{ij}^2\\
         = & \frac1\tau \sum_{i=1}^m g(\tilde u_{ii})+ \frac1{2\tau} \sum_{i\neq j} \tilde u_{ij}^2.
    \end{aligned}
\end{equation}
Here the inequality in \eqref{ineq:Gun1} is deduced from Lemma \ref{lem2}.
From \eqref{eq:rhs}, \eqref{eq:Gun} and \eqref{ineq:Gun1}, we have
\begin{equation}\label{ineq:term3}
\begin{aligned}
&\int_{\Omega} \frac 1 {\tau} \left\langle {\left((e^{2\tau}-1) \tilde{U}^n {(\tilde U^n)}^\top  +  I\right)^{-\frac12}} {e^\tau\tilde{U}^n}-\tilde{U}^{n} , \;\;
\tilde{U}^{n+1}-\tilde{U}^{n}  \right\rangle_F  \,dx\\
\le & \;\int_{\Omega} \left\langle G(\tilde U^{n}),I\right\rangle  \,dx  - \int_{\Omega} \left\langle G(\tilde U^{n+1}),I\right\rangle \,dx
+ \frac1{2\tau} \sum_{i\neq j} \tilde u_{ij}^2
+\frac 1 {2\tau} \sum_{i=1}^m (\tilde u_{ii}-\sigma_i)^2 \\
= & \; \int_{\Omega} \left\langle G(\tilde U^{n}),I\right\rangle  \,dx  - \int_{\Omega} \left\langle G(\tilde U^{n+1}),I\right\rangle \,dx
+\frac 1 {2\tau}  \int_{\Omega} \| \tilde U^{n+1}
-\tilde U^n\|_F^2 \,dx.
\end{aligned}
\end{equation}
Combining \eqref{eq:idFrob}, \eqref{ineq:term1}, and \eqref{ineq:term3}, it follows that
\begin{equation}
\begin{aligned}
&\widetilde E( {U}^{n+1})- \widetilde E( {U}^{n})\\
=& \int_\Omega \frac{1}{2\tau}\left\langle (1-e^{\tau \varepsilon^2\Delta})  {U^{n+1}},  {U^{n+1}} \right\rangle_F + \left\langle G(e^{\tau/2 \varepsilon^2\Delta} U^{n+1}),I\right\rangle_F \, dx\\
& -\int_\Omega \frac{1}{2\tau}\left\langle (1-e^{\tau \varepsilon^2\Delta})  {U^n},  {U^n} \right\rangle_F - \left\langle G(e^{\tau/2 \varepsilon^2\Delta} U^n),I\right\rangle_F \, dx \\
=& \frac 12 \int_{\Omega} 
\langle  (e^{-\tau \varepsilon^2\Delta}-1)\tilde{U}^{n+1},\;
\tilde{U}^{n+1} \rangle_F\, dx + \int_{\Omega} \left\langle G(\tilde U^{n+1}),I\right\rangle \,dx \\
& -\frac 12 \int_{\Omega} 
\langle  (e^{-\tau \varepsilon^2\Delta}-1)\tilde{U}^{n},\;
\tilde{U}^{n} \rangle_F \,dx - \int_{\Omega} \left\langle G(\tilde U^{n}),I\right\rangle  \,dx  
\\
\leq &-\frac 1 {2\tau}  \int_{\Omega} \| \tilde U^{n+1}
-\tilde U^n\|_F^2 \,dx -\frac 1{2\tau} \int_{\Omega} 
\langle  (e^{-\tau \varepsilon^2\Delta}-1)(\tilde{U}^{n+1}-\tilde{U}^n),\;
\tilde{U}^{n+1} -\tilde{U}^n\rangle_F\, dx\\
\leq &0.
\end{aligned}
\end{equation}
\end{proof}

\section{Uniform stability and convergence analysis}\label{sec:stability_convergence} 

The numerical solution of Strang splitting scheme for matrix-valued Allen--Cahn equation preserves the maximum principle unconditionally. 
We now state and prove the maximum principle result slightly different from the one in \cite{li2022stability}. 
\begin{lem}[Maximum principle]\label{lem:MP}
If $ \| U^0 \|_F\leq \sqrt m$ for any $x\in \overline\Omega$, then the solution of  Strang splitting method \eqref{eq:Strang} satisfies the maximum principle:
\begin{equation}
\begin{aligned}
 \| U^n \|_F \le  \sqrt m,\quad \forall n\geq 1.
\end{aligned}
\end{equation}
\end{lem}
\begin{proof}
    It is not difficult to verify that both the linear and nonlinear solution operators preserves the maximum principle: for any $\|V\|_F \leq \sqrt m$, then $\|\mathcal S_{\mathcal L}(\tau) V \|_F\leq \sqrt m$ and $\|\mathcal S_{\mathcal N}(\tau) V \|_F\leq \sqrt m$. 
\end{proof}

\subsection{Global-in-time $H^1$-stability}
We now establish uniform $H^1$-stability of the Strang splitting method for matrix-valued Allen--Cahn equation, that is based on the maximum principle, Lemma \ref{lem:MP} and thus
different from the scalar case in \cite{li2022stability}.

In the following content, the $H^k$-norm of a matrix-valued function $U$ is defined as
\begin{equation}\label{eq:HkMatrix}
\begin{aligned}
 \left\| U\right\|_{H^k(\Omega)} \coloneqq \left( \sum_{i,j=1}^m \left\|U_{ij}\right\|^2_{H^k(\Omega)}\right)^{\frac12}
\end{aligned}
\end{equation}
and $\| \nabla U \|_F$ is defined as
\begin{equation}
    \| \nabla U \|_F \coloneqq \left( \sum_{i,j=1}^m |\nabla U_{ij}|^2\right)^{\frac12}.
\end{equation}

\begin{lem}[Global-in-time $H^1$-stability of $\tilde U^n$]\label{lem:stability_Utilde}
If $ \| U^0\|_F^2 \leq  m$ for any $x\in \overline\Omega$, then the solution of Strang splitting method \eqref{eq:Strang} satisfies 
\begin{equation}
\begin{aligned}
 \left\|\| \nabla \tilde U^n \|_F \right\|_{L^2(\Omega)}^2\le  C,\quad\forall n\geq 0,
\end{aligned}
\end{equation}
where $C$ depends only on $\varepsilon$, $\left\| U^0\right\|_{H^1}$, $m$, and $\Omega$.
\end{lem}

\begin{proof}
Let $\widehat U^0$ be the Fourier transformation of $U^0$, defined by
\begin{equation}
    \left(\widehat U^0(k)\right)_{ij} = \widehat{U^0_{ij}}(k),\quad i,j = 1,2,\ldots, m,\mbox{ and } k \in \mathbb Z^d,
\end{equation}
where $U_{ij}^0$ is the $(i,j)$th element of $U^0$.
Using  the fact
\begin{equation}
    \frac{1-e^{-\tau \varepsilon^2|k|^2}}{\tau}\leq \varepsilon^2|k|^2,\quad \forall \tau>0,
\end{equation}
we have
\begin{equation}\label{eq:E0_1}
    \begin{aligned}
       & \int_\Omega \frac{1}{2\tau}\left\langle (1-e^{\tau \varepsilon^2\Delta})  {U^0},  {U^0} \right\rangle_F \\
   = &c_d \sum_{i,j=1}^m \sum_{0\neq k\in \mathbb Z^d}\frac{1-e^{-\tau \varepsilon^2|k|^2}}{2\tau} \left( \widehat U^0(k)\right)_{ij}^2
   \leq   c_d \sum_{i,j=1}^m \sum_{0\neq k\in \mathbb Z^d}\frac{\varepsilon^2|k|^2}{2} \left( \widehat U^0(k)\right)_{ij}^2 \\
   \leq & \sum_{i,j=1}^m \frac {\varepsilon^2} 2 \left\| U^0_{ij}(x) \right\|_{H^1}^2 
   = \frac{\varepsilon^2}2 \left\| U^0\right\|_{H^1}^2,
\end{aligned}
\end{equation}
where $c_d>0$ is a normalization constant depending on the spatial dimension $d$. 
Consider the singular value decomposition 
\begin{equation}
\tilde U^0 =e^{\frac{\tau \varepsilon^2}{2} \Delta}U^0 = \tilde P\tilde \Sigma \tilde Q
\end{equation}
with singular values $\{\tilde\sigma_i\}_{i=1}^m$ and orthogonal matrices $\tilde P$ and $\tilde Q$. 
Due to the fact
\begin{equation}
    \|\tilde U^0\|_F^2 \leq  \| U^0\|_F^2 \leq  m,
\end{equation}
then we have
\begin{equation}
    \tilde\sigma_i^2(x) \leq \sum_{i=1}^m \tilde\sigma_i^2(x)  =\|\tilde U^0\|_F^2 \leq  m.
\end{equation}
Therefore, we have
\begin{equation}\label{eq:E0_2}
\begin{aligned}
     &\int_\Omega \left\langle G(e^{\tau/2\varepsilon^2\Delta}U^0),I\right\rangle_F \, dx =   \int_\Omega \mbox{Tr}(G(e^{\tau/2\varepsilon^2\Delta}U^0)) \, dx
     =  \int_\Omega \mbox{Tr}(G(\tilde \Sigma)) \, dx\\
     = & \sum_{i=1}^m \int_\Omega\left[\frac{1}{2\tau}\tilde \sigma_i^2-\frac{e^\tau}{\tau(e^{2\tau}-1)} \left((1+(e^{2\tau}-1) \tilde\sigma_i^2)^{\frac12}-1\right) \right]\,dx\\
     \leq & \begin{cases}
     \displaystyle
          \sum_{i=1}^m \frac12 \int_\Omega \tilde \sigma_i^2\,dx, & \mbox{if }\tau\geq 1\\
      \displaystyle   \sum_{i=1}^m \int_\Omega \left[-\frac{e^\tau-1}{2\tau} \tilde\sigma_i^2+\frac{e^\tau(e^{2\tau}-1)}{8\tau} \tilde\sigma_i^4\right]\,dx,&\mbox{if } \tau <1
     \end{cases}\\
     \leq & 3m^3|\Omega|,
     \end{aligned}
\end{equation}
where $|\Omega|$ is the volume of $\Omega$.
Combining \eqref{eq:E0_1} and \eqref{eq:E0_2}, we then have
\begin{equation}
\begin{aligned}
    \widetilde E({U^0}) = &  \int_\Omega \frac{1}{2\tau}\left\langle (1-e^{\tau \varepsilon^2\Delta})  {U^0},  {U^0} \right\rangle_F + \left\langle G(e^{\tau/2 \varepsilon^2\Delta}U^0),I\right\rangle_F \, dx \leq \frac{\varepsilon^2}2 \left\| U^0\right\|_{H^1}^2 +3m^3|\Omega|.
    \end{aligned}
\end{equation}

From Lemma \ref{lem2} and Theorem \ref{thm:energy}, we have
\begin{equation}
\begin{aligned}
\frac{\varepsilon^2}2 \left\|\| \nabla \tilde U^n \|_F \right\|_{L^2(\Omega)}^2\le & 
\int_\Omega \frac 1 {2\tau}
\langle (e^{-\tau \varepsilon^2\Delta} -1) \tilde U^n, \, \tilde U^n \rangle_F \le \widetilde E({U^n})+\frac14 m |\Omega| \\
\leq & \widetilde E({U^0})+\frac14 m |\Omega| \leq C,
\end{aligned}
\end{equation}
where $C$ depends on $\varepsilon$, $\left\| U^0\right\|_{H^1}$, $m$, and $\Omega$.

\end{proof}

\begin{lem}[Properties of $\mathcal S_{\mathcal N}$]\label{lem:Uniform_SN}
If $ \| U^0 \|_F\leq \sqrt m$  for any $x\in \overline\Omega$ and $\| \nabla U^0\|_F \in L^2(\Omega)$, then for any time step $\tau>0$, the  solution $U_1=\mathcal S_{\mathcal N}(\tau) U^0$ to the nonlinear ODE \eqref{eq:nlODE} satisfies
\begin{equation}
     \left\|\| \nabla U_1\|_F \right\|_{L^2(\Omega)}\le  C e^{(1+3m) \tau}  \left\|\| \nabla U^0\|_F \right\|_{L^2(\Omega)},\quad\forall t\in[0,\tau],
\end{equation}
where $C$ depends only on $m$ and $d$.
Furthermore, if $ \| U^0\|_F^2 \leq  m$ and $ U^0 \in H^k(\Omega)$ with $k\geq 1$, then
\begin{equation}
\begin{aligned}
 \left\| U_1\right\|_{H^k(\Omega)}\le  C_1 e^{\alpha_{m,k}\tau}\left\| U^0\right\|_{H^k(\Omega)},\quad\forall t\in[0,\tau],
\end{aligned}
\end{equation}
where $C_1>0$ and $\alpha_{m,k}>0$ depend on $m$, $d$, and $k$.
\end{lem}

\begin{proof}

First, it is not difficult to obtain $ \| U_1\|_F^2 \leq  m$ for any $x\in \overline{\Omega}$.
Let $q = \partial_{x_{i}} U_1$, where $\partial_{x_{i}}$ denotes the partial derivative w.r.t. spatial variable $x_{i}$ with $1 \leq i\leq d$. 
Acting $\partial_{x_i}$ on both sides of \eqref{eq:nlODE}, we have
\begin{align} \label{eq:ode_q}
\begin{cases}
\partial_t q = q-q U_1^\top U_1 - U_1 q^\top U_1 -U_1 U_1^\top q, \\
q(0) = q_0\coloneqq \partial_{x_{i}} U^0.
\end{cases}
\end{align}
By taking $\mbox{Tr}(~\cdot ~q^\top )$ on both sides of the first equation in \eqref{eq:ode_q}, we have
\begin{equation}
    \begin{aligned}
    \frac12 \partial_t \| q\|_F^2 & = \| q\|_F^2 -\mbox{Tr}(q U_1^\top U_1 q^\top)-\mbox{Tr}(U_1 q^\top U_1q^\top )-\mbox{Tr}(U_1 q^\top U_1q^\top )\\
    & \leq \| q\|_F^2 + 3 \| U_1\|_F^2 \| q\|_F^2 
     \leq (1+ 3 m) \| q\|_F^2,
    \end{aligned}
\end{equation}
which yields that 
\begin{equation}
    \| q(t,x)\|_F^2 \leq e^{2(1+3m) \tau} \| q(0,x)\|_F^2, \quad \forall t\in[0,\tau].
\end{equation}
This indicates
\begin{equation}
\begin{aligned}
 \left\|\| \nabla U_1\|_F \right\|_{L^2(\Omega)}\le  C e^{(1+3m) \tau}  \left\|\| \nabla  U^0\|_F \right\|_{L^2(\Omega)},\quad\forall t\in[0,\tau],
\end{aligned}
\end{equation}
where $C$ depends only on $m$ and $d$.

By mathematical induction and similar arguments, it is not difficult to prove that 
\begin{equation}
\begin{aligned}
 \left\| U_1 \right\|_{H^k(\Omega)}\le  C_1 e^{\alpha_{m,k}\tau}\left\| U^0\right\|_{H^k(\Omega)},\quad\forall t\in[0,\tau],
\end{aligned}
\end{equation}
where $C_1$ and $\alpha_{m,k}$ depend on $m$, $d$, and $k$.

\end{proof}

Based on the stability of $\tilde U^{n}$ and the properties of $\mathcal S_{\mathcal N}$, we are ready to prove the global-in-time stability of $U^n$.

\begin{thm}[Global-in-time $H^1$-stability of $U^n$]\label{thm:stb}
    If $ \| U^0 \|_F\leq \sqrt m$  for any $x\in \overline\Omega$ and $\| \nabla U^0\|_F \in L^2(\Omega)$, then the solution of Strang splitting method \eqref{eq:Strang} satisfies 
\begin{equation}
\begin{aligned}
 \left\|\| \nabla U^n \|_F \right\|_{L^2(\Omega)}^2\le  C,\quad \forall n\geq 1,
\end{aligned}
\end{equation}
where $C$ depends only on $\varepsilon$, $\left\| U^0\right\|_{H^1}$, $m$, and $\Omega$.
\end{thm}
\begin{proof}
Note that for any $n\geq 1$,
\begin{equation}
    U^n = \mathcal S_{\mathcal L}(\tau/2) \mathcal S_{\mathcal N}(\tau)\tilde U^{n-1}.
\end{equation}
    According to Lemma \ref{lem:stability_Utilde},
    \begin{equation}
\begin{aligned}
 \left\|\| \nabla \tilde U^{n-1} \|_F \right\|_{L^2(\Omega)}^2\le  C,\quad \forall n\geq 1,
\end{aligned}
\end{equation}
where $C$ depends on $\varepsilon$, $\left\| U^0\right\|_{H^1}$, $m$, and $\Omega$.
If $0<\tau\leq 1$, according to Lemma \ref{lem:Uniform_SN} and the fact that $S_{\mathcal L}(\tau/2)$ is contractive in $L^2$-norm, we then have
\begin{equation}
\begin{aligned}
 \left\|\| \nabla U^n \|_F \right\|_{L^2(\Omega)}^2\le \left\|\| \nabla (\mathcal S_{\mathcal N}(\tau)\tilde U^{n-1}) \|_F \right\|_{L^2(\Omega)}^2 \leq C e^{2(1+3m) \tau}  \left\|\| \nabla \tilde U^n\|_F\right\|_{L^2(\Omega)}^2\leq C.
\end{aligned}
\end{equation}
If $\tau >1$, then we have 
\begin{equation}
\begin{aligned}
 \left\|\| \nabla U^n \|_F \right\|_{L^2(\Omega)}^2\le C \left\|\mathcal S_{\mathcal N}(\tau)\tilde U^{n-1} \right\|_{L^2(\Omega)}^2 \leq C,
\end{aligned}
\end{equation}
because $\|\mathcal S_{\mathcal N}(\tau)\tilde U^{n-1} \|^2_F\leq \|\tilde U^{n-1} \|^2_F\leq m$ holds.
\end{proof}

\subsection{Uniform bound of determinant}
We now state and prove a new result on the bound of determinant of matrix-field. 
\begin{lem}\label{lem:ineqa}
    For any sequence $\{a_i\in \mathbb R\}_{i=1}^m$ satisfying
    \begin{equation}
        \sum_{i=1}^m a_i^2 \leq a,
    \end{equation}
    the following inequality holds 
    \begin{equation}\label{ineq:a}
        \prod_{i=1}^m |a_i| \leq \left(\frac a m\right)^{\frac m 2}.
    \end{equation}
\end{lem}
\begin{proof}
    We prove this lemma by induction.
    When $m=1$, the proof is trivial. 
    Assume that the result holds for $m = m_0-1$. Consider the case $m=m_0$. Since
    \begin{equation}
        \sum_{i=1}^{m_0-1} a_i^2 \leq a-a_{m_0}^2,
    \end{equation}
    we obtain from the assumption that
    \begin{equation}
        \prod_{i=1}^{m_0-1} |a_i| \leq \left(\frac {a-a_{m_0}^2}{m_0-1}\right)^{\frac {m_0-1} 2}.
    \end{equation}
    As a consequence,
    \begin{equation}
         \prod_{i=1}^{m_0} |a_i| \leq \left(\frac {a-a_{m_0}^2}{m_0-1}\right)^{\frac {m_0-1} 2} |a_{m_0}| = \left(\left(\frac {a-a_{m_0}^2}{m_0-1}\right)^{ m_0-1} a_{m_0}^2\right)^{\frac12}.
    \end{equation}
    Consider the function
    \begin{equation}
        h(s) = \left(\frac {a-s}{m_0-1}\right)^{ m_0-1} s, \quad s\in [0,a].
    \end{equation}
    Letting
    \begin{equation}
        h'(s) = \left(\frac {a-s}{m_0-1}\right)^{ m_0-2}\left[ \frac {a-s}{m_0-1}- s \right] = 0,
    \end{equation}
    we have $s ={a}/{m_0}$. This implies that 
    \begin{equation}
        h(s)\leq h\left(\frac a{m_0}\right)= \left(\frac a{m_0}\right)^{m_0}
    \end{equation}
    and then
    \begin{equation}
         \prod_{i=1}^{m_0} |a_i| \leq \left(\frac a{m_0}\right)^{\frac{m_0}2}.
    \end{equation}
    Therefore, the inequality \eqref{ineq:a} holds for all $m\geq 1$.
\end{proof}
\begin{thm}[Uniform bound of determinant of $U^n$]\label{thm:det}
    If $ \| U^0 \|_F\leq \sqrt m$  for any $x\in \overline\Omega$, then the exact solution $U(t)$ to MAC equation \eqref{eq:matrixAC} and the numerical solution $U^n$ of Strang splitting method \eqref{eq:Strang} satisfy
\begin{equation}
\begin{aligned}
 \left|\det(U)\right| \leq 1\quad\mbox{and}\quad \left|\det(U^n)\right|\leq 1.
\end{aligned}
\end{equation}
\end{thm}
\begin{proof}
        Suppose that $U(t,x)$ has a SVD decomposition $U = P \Sigma Q^\top $ with orthogonal matrices $P$ and $Q$ and $\Sigma = \diag\{\sigma_1,\ldots,\sigma_m\}$. 
    Due to the fact $\|U\|_F^2 \leq  m$, we have 
    \begin{equation}
        \sum_{i =1}^m \sigma_i^2 \leq m.
    \end{equation}
    According to Lemma \ref{lem:ineqa}, we have
\begin{equation}
    \left|\det(U)\right| = \prod_{i=1}^m \left|\sigma_i\right|\leq 1.
\end{equation}
Similarly, we have $\left|\det(U^n)\right|\leq 1$.
\end{proof}

\subsection{Convergence analysis}
\begin{lem}\label{lem:Uniform_MAC}
If $ \| U^0 \|_F\leq \sqrt m$  for any $x\in \overline\Omega$ and $\| \nabla U^0\|_F \in L^2(\Omega)$ and $U^0 \in H^k(\Omega)$ with $k\geq 1$, then for any finite time $T>0$, the solution $U(t,x)$ to the MAC equation \eqref{eq:matrixAC} satisfies
\begin{equation}
\begin{aligned}
 \left\| U \right\|_{H^k(\Omega)}\le  C_2,\quad\forall t\in[0,T],
\end{aligned}
\end{equation}
where $C_2$ depends on $\left\| U^0\right\|_{H^k}$, $T$,  $m$, $k$, $d$, and $\Omega$.
\end{lem}

\begin{proof}
    The proof is similar to the one of Lemma \ref{lem:Uniform_SN} and is omitted here. 
\end{proof}

\begin{lem}[Dependency of $\mathcal S_{\mathcal N}$ on initial condition]\label{lem:dependV1V2}
    If $V_i\in \mathbb R^{m\times m}$ satisfy $\|V_i\|_F\leq \sqrt m$ for $i=1,2$, then
\begin{equation}
\|\mathcal S_{\mathcal N}(\tau) V_1-\mathcal S_{\mathcal N}(\tau) V_2\|_{L^2(\Omega)}\le  e^{(1+3m)\tau} \|V_1-V_2\|_{L^2(\Omega)}.
\end{equation}
\end{lem}
\begin{proof}
Consider 
\begin{align}
\begin{cases}
\partial_t U_1 = U_1-U_1U_1^\top U_1, \\
U_1(0) = V_1
\end{cases}
\end{align}
and
\begin{align}
\begin{cases}
\partial_t U_2 = U_2-U_2U_2^\top U_2, \\
U_2(0) = V_2.
\end{cases}
\end{align}
It is easy to check that $\|U_i\|_F\leq \sqrt m$, $i=1,2$.
Let $U = U_1-U_2 = \mathcal S_{\mathcal N}(t) V_1-\mathcal S_{\mathcal N}(t) V_2$.
Then we have
\begin{align}\label{eq:diffV1V2}
\begin{cases}
\partial_t U =U-(U_1U_1^\top U_1-U_2U_2^\top U_2), \\
U(0) = V_1-V_2.
\end{cases}
\end{align}
By taking $\mbox{Tr}(~\cdot ~U^\top )$ on both sides, we have
\begin{equation}
    \begin{aligned}
    \frac12 \partial_t \| U\|_F^2 & =  \|  U\|_F^2 -\mbox{Tr}\left((U_1U_1^\top U_1-U_2U_2^\top U_2)U^\top \right)\\
    & =  \|  U\|_F^2 -\mbox{Tr}\left((U_1U_1^\top U+U_1U^\top U_2+UU_2^\top U_2)U^\top \right)\\
    & \leq (1+ 3 m) \| U\|_F^2,
    \end{aligned}
\end{equation}
which yields that 
\begin{equation}
    \| U\|_F^2 \leq e^{2(1+3m) \tau} \|V_1-V_2\|_F^2, \quad \forall t\in[0,\tau].
\end{equation}
The proof is completed.
\end{proof}

\begin{thm}[Second-order convergence in time]\label{thm:conv}
Assume the initial data $U^0$ is sufficiently smooth and $ \| U^0 \|_F\leq \sqrt m$  for any $x\in \overline\Omega$.  Let $U(t)$ be the exact solution to \eqref{eq:matrixAC}  with initial data $U^0$. 
The numerical solution $U^n$ of Strang splitting method \eqref{eq:Strang} has second-order error estimate:
\begin{align} \label{Eu_1.23}
\sup_{n\ge 1, n\tau \le T}  \| U^n - U(t_n, \cdot ) \|_{L^2(\Omega)}
\le C  \tau^2,
\end{align}
where $t_n = n\tau$ and $C>0$ depends on ($\eps$, $U^0$, $T$).
\end{thm}

\begin{proof}
Denote by
\begin{equation}
\mathcal S(\tau) = \mathcal S_{\mathcal L}({\tau}/2)\mathcal S_{\mathcal N} (\tau) 
\mathcal S_{\mathcal L}({\tau}/2).
\end{equation} 
We first analyze the local truncation error for the propagator $\mathcal S(\tau)$. 
Let $\mathcal L = \varepsilon^2\Delta$ and assume that $V$ is a sufficiently regular initial condition.
Straight computation gives 
\begin{equation}
\begin{aligned}
&B= \mathcal S_{\mathcal L}({\tau}/2) V = V + \frac{\tau}2 \mathcal L V
+ \frac{\tau^2} 8 \mathcal L^2 V +\mathcal O (\tau^3),\\
&W = \mathcal S_{\mathcal N}(\tau) B = B + \tau P + \frac 12 \tau^2 (P - PB^\top B-BP^\top B-BB^\top P ) +  \mathcal O(\tau^3),\\
& \mathcal S_{\mathcal L}({\tau}/2) W = W + \frac{\tau}2 \mathcal L W
+ \frac{\tau^2} 8 \mathcal L^2 W +\mathcal O (\tau^3),
\end{aligned}
\end{equation}
where $P=B-BB^\top B$ and the constant coefficient of $\mathcal O(\tau^3)$ depends on the regularity of $V$.
Then we have 
\begin{equation}
\begin{aligned}
 \mathcal S(\tau)  V 
= & \mathcal S_{\mathcal L}(\tau/2) \mathcal S_{\mathcal N}(\tau) \left[V + \frac{\tau}2 \mathcal L V
+ \frac{\tau^2} 8 \mathcal L^2 V +\mathcal O (\tau^3) \right] \\
= &  \mathcal S_{\mathcal L}( {\tau}/2) \Bigl[ V+\tau\left(\frac12\mathcal L+1-VV^\top \right)V + \frac12\tau^2\Bigl(\frac14\mathcal L^2+\mathcal L \\
&\quad +(1-4VV^\top +3VV^{\top}VV^{\top}) -(\mathcal LV)V^{\top} - V(\mathcal L V)^{\top}-VV^{\top}\mathcal L\Bigr)V + \mathcal O(\tau^3)\Bigr] \\
=& V +\tau(\mathcal L+1-VV^{\top}) V + \frac12\tau^2 \Bigl(\mathcal L^2 V+2 \mathcal L V +(1-4VV^{\top}+3VV^{\top}VV^{\top})V \\
&\quad -(\mathcal LV)V^{\top}V - V(\mathcal L V)^{\top}V-VV^{\top}(\mathcal L V) -\mathcal L(VV^{\top}V) \Bigr) + \mathcal O(\tau^3). \label{Eu_2.32}
\end{aligned}
\end{equation}
We now turn to the expansion  of the exact PDE solution. 
Let $\mathcal T(\tau)$ be the exact solution operator at $t=\tau$ to \eqref{eq:matrixAC} with initial data $V$. We then have
\begin{equation}
\begin{aligned}\label{Eu_2.33}
\mathcal T(\tau) V & = \mathcal S_{\mathcal L}(\tau) V + \int_0^{\tau} 
\mathcal S_{\mathcal  L}(\tau -s) \left[U(s)-U(s)U^{\top}(s) U(s)\right] ds  \\
& =\mathcal S_{\mathcal L}(\tau)V + \int_0^{\tau} ( 1+ (\tau-s) \mathcal L) \left[U(s)-U(s)U^{\top}(s) U(s)\right] ds + \mathcal O(\tau^3)  \\
& = \mathcal S_{\mathcal L}(\tau)V+ \int_0^{\tau} 
\Bigl(  V -VV^{\top} V + s( Q-QV^{\top}V-VQ^{\top}V-VV^{\top}Q)   
\Bigr) ds \\
&\quad + 
\int_0^{\tau} (\tau-s)\mathcal L (V-VV^{\top}V) ds + \mathcal O(\tau^3)  \\
& = \mathcal S_{\mathcal L}(\tau)V+ \tau (V -VV^{\top} V )
+ \frac {\tau^2}2 \Bigl(  \mathcal L (V-VV^{\top}V)\\
&\quad +(Q-QV^{\top}V-VQ^{\top}V-VV^{\top}Q) \Bigr)+ \mathcal O(\tau^3) \\
& = V +\tau (\mathcal L V+ V -VV^{\top} V )
+ \frac {\tau^2}2 \Bigl( \mathcal L^2 V+ \mathcal L (2V-VV^{\top}V)\\
&\quad +(1-4VV^{\top}+3VV^{\top}VV^{\top})V  -(\mathcal LV)V^{\top}V - V(\mathcal L V)^{\top}V-VV^{\top}(\mathcal L V) \Bigr)+ \mathcal O(\tau^3), \\
\end{aligned}
\end{equation}
where $Q = \mathcal L V +V -VV^{\top} V$.
Clearly \eqref{Eu_2.32} and \eqref{Eu_2.33} have the same form in  $\mathcal O(\tau^3)$.

Since the initial data $u^0$ is sufficiently smooth, according to Lemma \ref{lem:Uniform_MAC}, we have 
\begin{equation}\label{J18_22a}
\| U(t_n) \|_{H^{k} (\Omega) }\le C,
\end{equation} 
where $C$ depends on $\left\| U^0\right\|_{H^k}$, $T$,  $m$, $k$, $d$, and $\Omega$.
By the triangle inequality, we have
\begin{equation}
\begin{aligned}
&\|U^n-U(t_n)\|_{L^2(\Omega)}\\
& \le 
\|\mathcal S(\tau) U^{n-1}-\mathcal S(\tau) U(t_{n-1})\|_{L^2(\Omega)}
+ \|\mathcal S(\tau) U(t_{n-1})-\mathcal T(\tau) U(t_{n-1})\|_{L^2(\Omega)}.
\end{aligned}
\end{equation}
By using \eqref{Eu_2.32}, \eqref{Eu_2.33}, and \eqref{J18_22a}, we have
\begin{equation}
\|\mathcal S(\tau) U(t_{n-1})-\mathcal T(\tau) U(t_{n-1})\|_{L^2(\Omega)}\le B_1\tau^3,
\end{equation} 
where $B_1>0$ is independent of $\tau$. 
According to Lemma \ref{lem:dependV1V2},
\begin{equation}
\begin{aligned}
&\|\mathcal S(\tau) U^{n-1}-\mathcal S(\tau) U(t_{n-1})\|_{L^2(\Omega)} \\
&\le \|\mathcal S_{\mathcal N}(\tau) \mathcal S_{\mathcal L}(\tau/2) U^{n-1}-\mathcal S_{\mathcal N}(\tau) \mathcal S_{\mathcal L}(\tau/2) U(t_{n-1})\|_{L^2(\Omega)}\\
& \le e^{(1+3m)\tau} \| \mathcal S_{\mathcal L}(\tau/2) U^{n-1}- \mathcal S_{\mathcal L}(\tau/2)U(t_{n-1})\|_{L^2(\Omega)}\\
&\le  e^{(1+3m)\tau} \|U^{n-1}-U(t_{n-1})\|_{L^2(\Omega)}.
\end{aligned}
\end{equation}
It follows that 
\begin{equation}
\|U^n-U(t_n)\|_{L^2(\Omega)} \le e^{(1+3m)\tau} \|U^{n-1}-U(t_{n-1})\|_{L^2(\Omega)} + B_1\tau^3. 
\end{equation}
As a consequence, we have
\begin{equation}
\begin{aligned}
\|U^n-U(n\tau)\|_{L^2(\Omega)} \le & e^{(1+3m)\tau} \|U^{n-1}-U(t_{n-1})\|_{L^2(\Omega)} + B_1\tau^3\\
\le & e^{(1+3m)2\tau} \|U^{n-2}-U((n-2)\tau)\|_{L^2(\Omega)} + B_1\tau^3 (1+e^{(1+3m)\tau})\\
\vdots & \\
\le & e^{(1+3m)n\tau} \|U^{0}-U(0)\|_{L^2(\Omega)} + B_1\tau^3 \sum_{j=0}^{n-1} e^{(1+3m)j\tau}\\
\le & B_1 T e^{(1+3m)T} \tau^2.
\end{aligned}
\end{equation}
The proof is completed.

\end{proof}

\section{Connections to diffusion generated methods}\label{sec:connections}
In the study of the unconditional energy dissipation, we constructed a modified energy $\widetilde E({U})$ as defined in \eqref{eq:modifiedenergy}. In this section, we start with the modified energy and derive the same method through the optimization point of view. We then make a strong connection with the diffusion generated methods when taking the limit as $\varepsilon \rightarrow 0$.

\subsection{Another derivation of the scheme from the optimization of the modified energy $\widetilde E({U})$.}
In this section, we first consider to generate a minimizing sequence from the optimization approach for the modified energy $\widetilde E({U})$. 
Rewrite $\widetilde E({U}) = \widetilde E_1({U})-\widetilde E_2({U})$ where 
 \begin{align}
\widetilde E_1({U}) =\int_\Omega \frac{1}{2\tau}\left\langle UU^{\top}+\frac{2e^\tau}{e^{2\tau}-1},  I \right\rangle_F \, dx 
\end{align}
and 
\begin{align}
\widetilde E_2({U}) =\int_\Omega \frac{1}{\tau}\left\langle \frac{e^\tau}{e^{2\tau}-1}\left(I+(e^{2\tau}-1)(e^{\frac{\tau \varepsilon^2\Delta}{2}}U)(e^{\frac{\tau \varepsilon^2\Delta}{2}}U)^{\top}\right)^{\frac{1}{2}},  I \right\rangle_F \, dx .
\end{align}

The Fr\'echet derivative of $\widetilde E_2({U})$ can be derived as follows.

\begin{lem}
The Fr\'echet derivative of $\widetilde E_2({U})$ with respect to $U$ is 
\[\frac{\delta \widetilde E_2({U})}{\delta U} = \dfrac{e^\tau}{\tau}e^{\frac{\tau \varepsilon^2\Delta}{2}}\left[\left(I+(e^{2\tau}-1)(e^{\frac{\tau \varepsilon^2\Delta}{2}}U)(e^{\frac{\tau \varepsilon^2\Delta}{2}}U)^{\top}\right)^{-\frac{1}{2}}e^{\frac{\tau \varepsilon^2\Delta}{2}}U\right].\]

\end{lem}

\begin{proof}

Assume the singular value decomposition  $e^{\frac12\tau\varepsilon^2\Delta} U = U_0 \Sigma V_0^{\top}$. Let $\varphi\in \mathbb R^{m\times m}$ be any given matrix-valued function. Then, we can calculate
\begin{equation}
    \begin{aligned}
    &\lim_{\gamma \rightarrow 0}\frac{1}{\gamma}\left[\widetilde E_2(U+\gamma\varphi) - \widetilde E_2(U)\right]\\
    = & \lim_{\gamma \rightarrow 0} \Bigg\{\frac{e^\tau}{\gamma \tau(e^{2\tau}-1)}\int_\Omega \left\langle \left[I+(e^{2\tau}-1)\left(e^{\frac{\tau \varepsilon^2\Delta}{2}}(U+\gamma\varphi)\right)\left(e^{\frac{\tau \varepsilon^2\Delta}{2}}(U+\gamma\varphi) \right)^{\top}\right]^{\frac{1}{2}},  I \right\rangle_F \, dx\\
    &-\frac{e^\tau}{\gamma\tau(e^{2\tau}-1)}\int_\Omega \left\langle \left[I+(e^{2\tau}-1)\left(e^{\frac{\tau \varepsilon^2\Delta}{2}}U\right)\left(e^{\frac{\tau \varepsilon^2\Delta}{2}}U \right)^{\top}\right]^{\frac{1}{2}},  I \right\rangle_F\, dx \Bigg\}\\
       = &\lim_{\gamma \rightarrow 0}\Bigg\{ \frac{e^\tau}{\gamma\tau(e^{2\tau}-1)}\int_\Omega \left\langle \left[I+(e^{2\tau}-1)\left(\Sigma^2+\gamma(\tilde\varphi\Sigma + \Sigma \tilde\varphi^{\top}) + \gamma^2 \tilde\varphi\tilde\varphi^{\top}\right)\right]^{\frac{1}{2}},  I \right\rangle_F\, dx\\
    &-\frac{e^\tau}{\gamma\tau(e^{2\tau}-1)}\int_\Omega \left\langle \left[I+(e^{2\tau}-1)\Sigma^2\right]^{\frac{1}{2}},  I \right\rangle_F\, dx \Bigg\}\\
    =&\frac{e^\tau}{\tau(e^{2\tau}-1)} \int_\Omega \left\langle \frac12 \left[I+(e^{2\tau}-1)\Sigma^2\right]^{-\frac{1}{2}}(e^{2\tau}-1)(\tilde\varphi\Sigma + \Sigma \tilde\varphi^{\top}),  I \right\rangle_F\, dx\\
    =&\frac{e^\tau}{\tau} \int_\Omega \left\langle \left[I+(e^{2\tau}-1)\Sigma^2\right]^{-\frac{1}{2}} \Sigma,  \tilde\varphi \right\rangle_F\, dx\\
    =&\frac{e^\tau}{\tau} \int_\Omega \left\langle e^{\frac{\tau \varepsilon^2\Delta}{2}}\left[\left(I+(e^{2\tau}-1)(e^{\frac{\tau \varepsilon^2\Delta}{2}}U)(e^{\frac{\tau \varepsilon^2\Delta}{2}}U)^{\top}\right)^{-\frac{1}{2}}e^{\frac{\tau \varepsilon^2\Delta}{2}}U\right], \varphi \right\rangle_F\, dx,
    \end{aligned}
\end{equation}
where $\tilde \varphi = U_0^{\top} (e^{\frac12\tau\varepsilon^2\Delta} \varphi) V_0$ and we use the fact $$\lim_{\gamma\rightarrow 0} \left[(X_1+\gamma X_2)^{\frac12} -X_1^{\frac12} \right]= \frac12 X_1^{-\frac12} X_2.$$
The proof is completed.

\end{proof}

The convexity of $\widetilde E_2$ can be further shown in the following lemma.

\begin{lem}[Convexity of $\widetilde E_2$]
Let \( A \in \mathbb{R}^{m \times n} \). The function 
\[
h(A) = \mathrm{Tr}\left( \left( I + AA^\top \right)^{1/2} \right)
\]
is convex in \( A \). Consequently, $\widetilde E_2 (U)$ is convex w.r.t. $U$. 
\end{lem}

\begin{proof}
Define the augmented matrix \( B = \begin{bmatrix} I & A \end{bmatrix}^\top \), where \( I \) is the identity matrix. Note that
\[
B^\top B = I + AA^\top,
\]
and the nuclear norm (sum of singular values) of \( B \) is
\[
\| B \|_* = \mathrm{Tr}\left( \left( B^\top B \right)^{1/2} \right) = \mathrm{Tr}\left( \left( I + AA^\top \right)^{1/2} \right) = h(A).
\]
The nuclear norm \( \| \cdot \|_* \) is a well-known convex function over matrices. Since \( B \) depends affinely on \( A \), the composition \( h(A) = \| B \|_* \) preserves convexity. Therefore, \( h(A) \) is convex in \( A \).
\end{proof}

Now, based on above properties on the decomposition of $\widetilde E({U})$, we consider to generate the following sequence based on an initial $U^0$,
\[U^1, U^2, \cdots, U^n, U^{n+1}, \cdots\]
such that
\begin{align}
U^{n+1} = \arg\min\limits_{U\in \mathbb R^{m\times m}} \widetilde E_1({U})-\mathcal{L}_2(U,U^n)\label{min:linearproblem}
\end{align}
where
\begin{align*}
\mathcal{L}_2(U,U^n) =E_2(U^n)+ \int_\Omega\left\langle \left.\frac{\delta \widetilde E_2({U})}{\delta U}\right|_{U=U^n}, U-U^n\right\rangle \, dx
\end{align*}
is the linearization of $\widetilde E_2({U})$ at $U^n$.

Because $\mathcal{L}_2(U,U^n)$ is linear in $U$ and $\widetilde E_1({U})$ is convex with respect to $U$, straightforward calculation of the Fr\'echet derivative yields the solution of \eqref{min:linearproblem}
\[ U^{n+1} = e^\tau e^{\frac{\tau \varepsilon^2\Delta}{2}}\left[\left(I+(e^{2\tau}-1)(e^{\frac{\tau \varepsilon^2\Delta}{2}}U^n)(e^{\frac{\tau \varepsilon^2\Delta}{2}}U^n)^{\top}\right)^{-\frac{1}{2}}e^{\frac{\tau \varepsilon^2\Delta}{2}}U^n\right]\]
which is exactly the Strang operator splitting scheme as discussed in \eqref{eq:Strang}.

We then give the proof of Theorem~\ref{thm:energy} from the optimization point of view.

\begin{proof}[Proof of Theorem~\ref{thm:energy}]
Denote $$\widetilde E^L_1({U})= \widetilde E_1({U})-\mathcal{L}_2(U,U^n)$$ and one can have
\begin{align}
\widetilde E^L_1({U^n}) = \widetilde E({U^n}),\\
\widetilde E^L_1({U^{n+1}}) \leq \widetilde E({U^n}),\label{eq:E1L}
\end{align}
where \eqref{eq:E1L} is a direct consequence that $U^{n+1}$ is a minimizer of a convex functional.
To prove the energy dissipation, one only needs to show that 
\[\widetilde E({U^{n+1}}) \leq \widetilde E^L_1({U^{n+1}})\]
which is equivalent to
\begin{align}
-\widetilde E_2({U^{n+1}}) \leq -\mathcal{L}_2(U^{n+1},U^n).\label{eq:linearless}
\end{align}
Considering the fact that $-\widetilde E_2({U^{n+1}})$ is concave and the graph of a concave functional always locates below its linearization, one can directly obtain \eqref{eq:linearless}.
\end{proof}

\subsection{The case when $\varepsilon\rightarrow 0$.}

In this section, we study the case of the Strang operator splitting scheme as $\varepsilon \rightarrow 0$. To be consistent with the framework of diffusion generated methods introduced in \cite{osting2020diffusion}, we consider the re-scaled equation with $\tilde t = \varepsilon^2 t $ to \eqref{eq:matrixAC} and thus have 
\begin{equation} \label{eq:resacledmatrixAC}
\partial_{\tilde t} U = \Delta U  + \frac{1}{\varepsilon^2}(U - UU^\top  U).
\end{equation}
In this time scale, $\mathcal S_{\mathcal N}\left( \tilde t \right)$ can be written into
\begin{align} \label{3t1}
\mathcal S_{\mathcal N}(\tilde t) U^0 := \left((e^{2\tilde t/\varepsilon^2}-1) U^0(U^0)^\top  + \mathrm I\right)^{-\frac12} e^{\tilde t/\varepsilon^2} U^0,\qquad \tilde t>0.
\end{align}

The following lemma shows that the operator $ \mathcal S_{\mathcal N}\left( \tau \right)$ converges to a projection operator to the orthogonal matrix group as $\varepsilon\rightarrow 0$.

\begin{lem}
Assume $U^0$ is nonsigular, as $\varepsilon\rightarrow 0$, $ \mathcal S_{\mathcal N}\left( \tilde t \right)$ defined in \eqref{3t1} converges to $ \mathcal P \left( \tilde t \right)$ with \begin{equation}\label{def:P}
    \mathcal P \left( \tilde t  \right) U^0 := P Q^\top,
\end{equation}
where $P$ and $Q$ are obtained from the singular value decomposition $U^0  = P \Sigma Q^\top$.
\end{lem}

\begin{proof}
Using $U^0 = P \Sigma Q^\top$, $\mathcal S_{\mathcal N}(\tilde t) U^0$ can be written into 
\[\mathcal S_{\mathcal N}(\tilde t) U^0 = P (I+(e^{2\tilde t/\varepsilon^2}-1)\Sigma^2)^{-\frac12} e^{\tilde t/\varepsilon^2} \Sigma Q^\top .\]
Because $\Sigma$ is a diagonal matrix will all positive entries $\sigma_{ii}$, one can have the limit of each diagonal entry in $(I+(e^{2\tilde t/\varepsilon^2}-1)\Sigma^2)^{-\frac12} e^{\tilde t/\varepsilon^2} \Sigma$:
\[\lim_{\varepsilon \rightarrow 0} \dfrac{e^{\tilde t/\varepsilon^2} \sigma_{ii}}{(1+(e^{2\tilde t/\varepsilon^2}-1)\sigma_{ii}^2)^{\frac12}} = 1,\]
implying the convergence of $ \mathcal S_{\mathcal N}\left( \tilde t \right)$ to $ \mathcal P \left( \tilde t \right)$.
\end{proof}

Then, the scheme with $\varepsilon\rightarrow 0$ can be further written as the diffusion generated scheme introduced in \cite{osting2020diffusion} in the framework of Strang operator splitting.

\begin{rem}
The Strang operator splitting scheme \eqref{eq:Strang} is second-order accurate in time when $\varepsilon$ is fixed. In the limit as $\varepsilon \rightarrow 0$, the exact solution converges to the mean curvature flow. The diffusion-generated algorithms can be formulated in a manner analogous to the Strang operator splitting scheme, even for the scalar case \cite{xu2017efficient}. However, numerical observations using spectrally accurate time-stepping algorithms indicate that the temporal accuracy of the diffusion-generated method for approximating the mean curvature flow remains first order \cite{jiang2018efficient}. This implies that the Strang operator splitting scheme may not be asymptotically compatible \cite{tian2014asymptotically}.
\end{rem}

\section{Numerical tests}\label{sec:num}
In this section, we illustrate some numerical results on the energy stability, determinant bound preservation and interfacial dynamics of Strang splitting method for MAC equation.

\subsection{Comparison between Strang splitting and thresholding methods}

We first present a comparative study of two numerical approaches for solving the rescaled matrix-valued Allen--Cahn equation \eqref{eq:resacledmatrixAC}: the Strang operator splitting method and the thresholding method \cite{osting2020diffusion}:
\begin{equation}
     U^{n+1}  = \mathcal S_{\mathcal L}\left( \tau/2\right)  \mathcal{P}\left( \tau \right)\mathcal S_{\mathcal L}\left( \tau/2\right) U^n
\end{equation}
where the projection $\mathcal{P}$ is defined in \eqref{def:P}.

The spatial domain $[-L/2, L/2]^2$ with $L=2\pi$ is discretized using $N=256$ uniform grid points in each direction, yielding spatial resolution $h = 2\pi/N$. Temporal integration employs a fixed time step $\tau =0.01$ up to final time $t_{\max}=2$. The Fourier spectral method is employed to solve the linear evolution $e^{\tau \Delta}$ with wavevectors $k_x, k_y \in [-N/2, N/2-1]$.

The initial matrix field $U^0 = v(x,y,0) \in \mathbb{R}^{2\times 2}$ combines angular modulation and radial confinement:
\begin{equation}
\begin{aligned}
    v_{11} &= \cos\alpha, \quad v_{12} = -\chi \sin\alpha  +(1-\chi) \sin\alpha, \\
v_{21} &= \sin\alpha, \quad v_{22} = \chi\cos\alpha   - (1-\chi) \cos\alpha,
\end{aligned}
\end{equation}
where
\begin{align*}
\alpha(x,y) &= \frac{\pi}{2}\sin(x+y),\quad
\chi(r,\theta)= \begin{cases} 
1 & r < 2\pi(0.3 + 0.06\sin6\theta) \\
0 & \text{otherwise}
\end{cases} 
\end{align*}
with $r(x,y) = \sqrt{x^2 + y^2}$ and $
\theta(x,y) = \arctan(x/y)$.

Due to the rescaling technique, the energy for the solution to \eqref{eq:resacledmatrixAC} is also rescaled as 
$$E(t) =\frac{1}{2}\int_\Omega \left\|\nabla U \right\|_F^2\ dx + \frac{1}{4\varepsilon^2}\int_\Omega \|I - UU^\top\|_F^2\ dx .$$
Moreover, the difference between the numerical solutions, $U^n_{\text{Strang}}$ and $U^n_{\text{Thres}}$, obtained from Strang splitting and thresholding methods respectively is characterized by
\begin{equation}\label{eq:dif}
\int_\Omega \|U^n_{\text{Strang}} - U^n_{\text{Thres}}\|_F^2\ dx.
\end{equation}

\begin{figure}[!h]
    \centering
\includegraphics[width=1\linewidth]{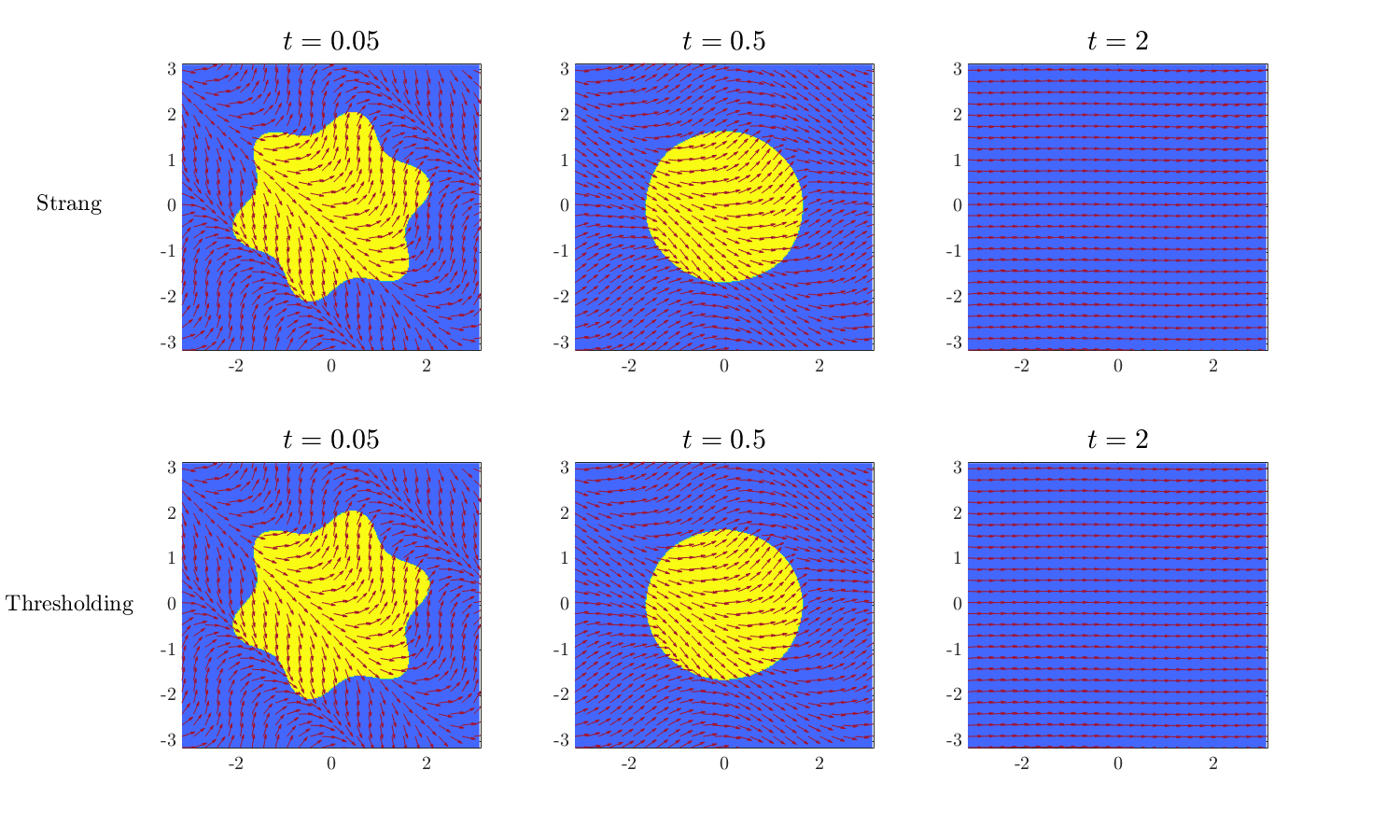}
    \caption{Interfacial snapshots at different time, where the region of $\det(U)>0$ is colored in yellow and the region of $\det(U)<0$ is colored in blue. }
    \label{fig:ex2sol}
\end{figure}

\begin{figure}[!h]
    \centering
\includegraphics[width=0.48\linewidth]{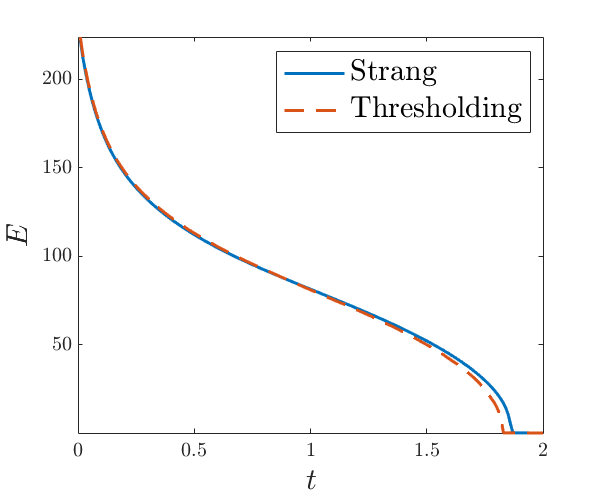}
\includegraphics[width=0.48\linewidth]{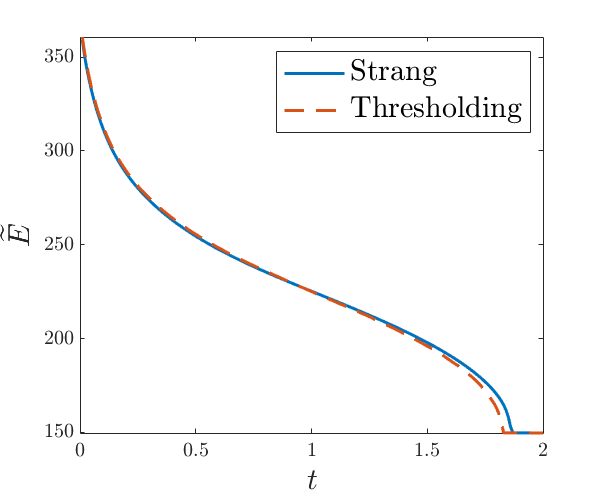}
    \caption{Evolution of original energy $E$ (left) and modified energy $\tilde E$ (right) for Strang splitting method and thresholding method.}
    \label{fig:ex2energy}
\end{figure}

\begin{figure}[!h]
    \centering
\includegraphics[width=0.48\linewidth]{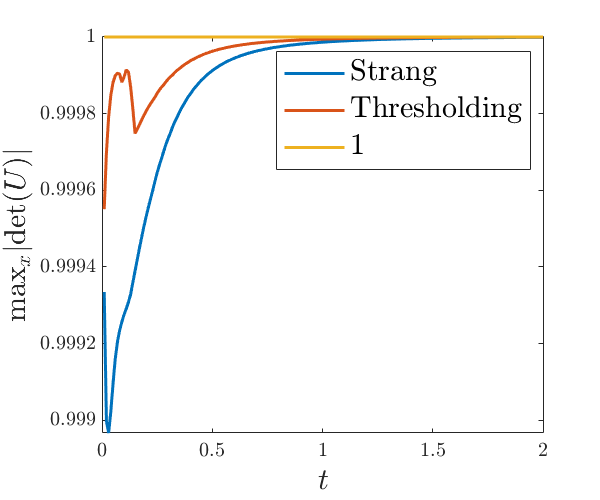}
\includegraphics[width=0.48\linewidth]{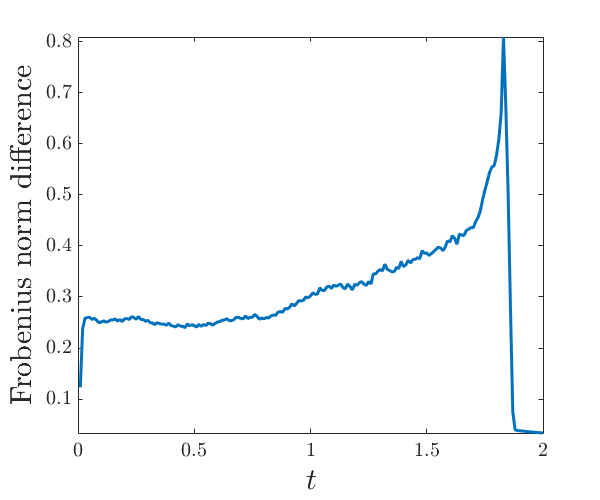}
    \caption{Evolution of determinant bound $\max _x|\operatorname{det}(U)|$ for the Strang splitting method and thresholding method (left); evolution of the Frobenius-norm difference \eqref{eq:dif} between $U^n_{\text{Strang}}$ and $U^n_{\text{Thres}}$.}
    \label{fig:ex2det_diff}
\end{figure}

From Figure~\ref{fig:ex2sol}, both methods generate similar evolution patterns. Moreover, the energy decay rates in Figure~\ref{fig:ex2energy} nearly coincide, suggesting comparable dissipation properties.
As observed in Figure \ref{fig:ex2det_diff}, the Strang method and the thresholding method both maintain maximum determinant bound $1$, and the Frobenius norm difference between these two methods can reach up to $0.8$ by $t=2$.
In summary, the Strang splitting demonstrates nice structure preservation at comparable computational cost, making it preferable for long-time simulations. Meanwhile, the thresholding method offers an alternative implementation.

\subsection{2D MAC with random initial data}
The dynamical evolution of the 2-dimensional MAC \eqref{eq:matrixAC} is simulated using a Fourier pseudo spectral method combined with an operator-splitting time integration scheme. The spatial domain is discretized over a $256 \times 256$ grid, while the temporal evolution employs a step size of $\tau =0.1$ up to $t_{\max }=20$. Each component of the initial condition $U^0$ is taken to be uniformly random in $[0,0.1]$.

In Figure \ref{fig:ex1sol}, the visualization of interfacial evolution reveals the emergence of topological defects and their subsequent coarsening dynamics, driven by competition between gradient energy minimization and nonlinear interactions. The quiver plots superimposed on the phase illustrate the orientation of the first column of $U$. 

In Figure \ref{fig:ex1energy_det}, the energy functional $E(t)$, comprising interfacial and gradient contributions, exhibits monotonic decay, confirming the energy dissipation property for MAC. The modified energy $\widetilde{E}(t)$, derived from the splitting scheme, closely aligns with $E(t)$, validating numerical stability. Notably, the determinant constraint $\max _x|\operatorname{det}(U)|$ remains bounded below unity, ensuring the preservation of the matrix field's physical admissibility.
\begin{figure}[!h]
    \centering
\includegraphics[width=1\linewidth]{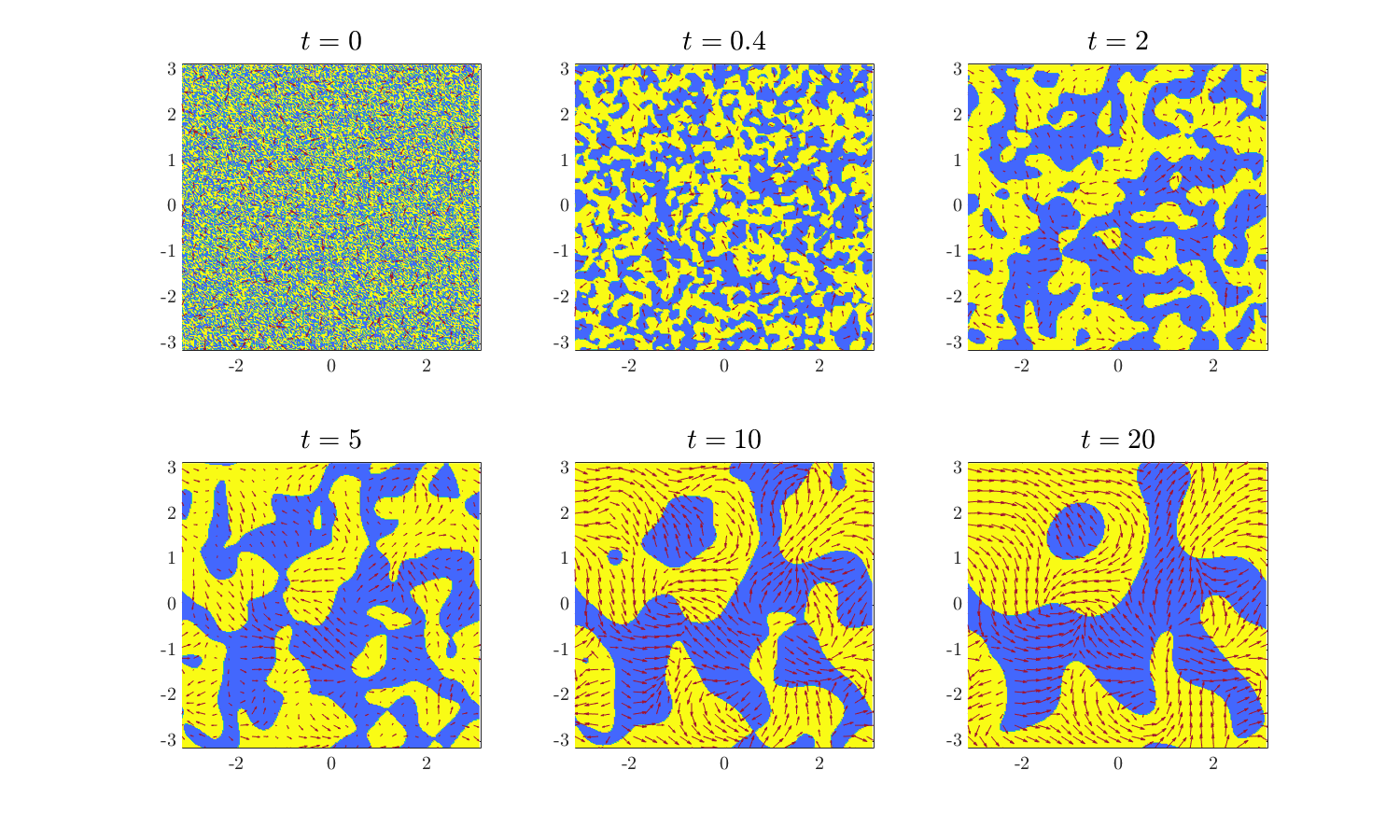}
    \caption{Interfacial snapshots at different time, where the region of $\det(U)>0$ is colored in yellow and the region of $\det(U)<0$ is colored in blue. }
    \label{fig:ex1sol}
\end{figure}

These results demonstrate the robustness of the Strang splitting method in capturing complex microstructure evolution for MAC equation under non-equilibrium conditions.

\begin{figure}[!h]
    \centering
\includegraphics[width=0.48\linewidth]{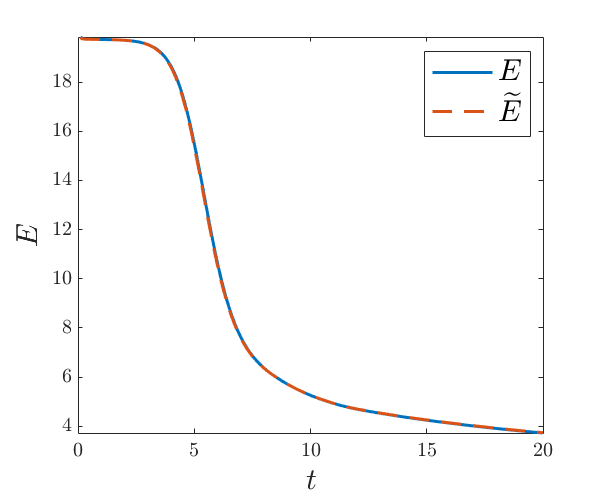}
\includegraphics[width=0.48\linewidth]{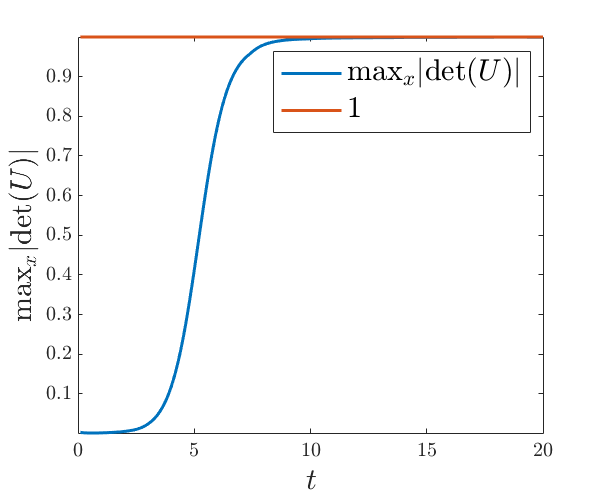}
    \caption{Evolution of original energy $E$ and modified energy $\tilde E$ (left); evolution of $\max _x|\operatorname{det}(U)|$ for the Strang splitting method (right).}
    \label{fig:ex1energy_det}
\end{figure}

\section{Conclusion and discussions}\label{sec:conclusion}
In this paper, we have presented a comprehensive analysis of the Strang splitting method for the matrix-valued Allen--Cahn equation. More precisely, we demonstrated in Theorem \ref{thm:energy} that the energy dissipation law holds for any time step size, removing the previous restriction on the time step in \cite{li2022stability2}. This was achieved through a novel trace inequality for matrix square roots.
We established the global-in-time $H^1$-stability of the numerical solution, ensuring that the solution remains bounded in the $H^1$ norm for all time.
We proved in Theorem \ref{thm:det} that the determinant of the matrix field is uniformly bounded by 1, which is crucial for maintaining the orthogonality of the matrix-valued solution.
In addition, we provided a rigorous second-order convergence analysis in time, showing that the numerical solution converges to the exact solution with second-order accuracy in Theorem \ref{thm:conv}.
Finally, we explored the connection between the Strang splitting method and diffusion-generated methods, particularly in the limit as $\varepsilon\rightarrow 0$. This provides a deeper understanding of the method's behavior in different regimes.

In future work, we plan to investigate the stability and convergence properties of higher order (than two) splitting methods under different boundary conditions and with different regularity assumptions for the matrix-valued Allen--Cahn equations.

\section*{Acknowledgment}
C. Quan was partially supported by the National Natural Science Foundation of China (Grant No. 12271241), Guangdong Provincial Key Laboratory of Mathematical Foundations for Artificial Intelligence (2023B1212010001), Guangdong Basic and Applied Basic Research Foundation (Grant No. 2023B1515020030), and Shenzhen Science and Technology Innovation Program (Grant No. JCYJ20230807092402004). T. Tang was partially supported by Science Challenge Project (Grant No. TZ2018001) and NSFC 11731006 and K20911001. D. Wang was partially supported by National Natural Science Foundation of China (Grant No. 12422116), Guangdong Basic and Applied Basic Research Foundation (Grant No. 2023A1515012199), Shenzhen Science and Technology Innovation Program (Grant No. JCYJ20220530143803007, RCYX20221008092843046), and Hetao Shenzhen-Hong Kong Science and Technology Innovation Cooperation Zone Project (No.HZQSWS-KCCYB-2024016).

\printbibliography

\end{document}

